\documentclass[11pt]{amsart}

\usepackage{amsthm, amsfonts, amssymb, amscd}
\usepackage[pagebackref,colorlinks]{hyperref}

\theoremstyle{definition}
\newtheorem{ntn}{Notation}[section]

\theoremstyle{plain}
\newtheorem{cor}[ntn]{Corollary}
\newtheorem{lem}[ntn]{Lemma}
\newtheorem{prp}[ntn]{Proposition}
\newtheorem{thm}[ntn]{Theorem}

\theoremstyle{remark}

\newtheorem{rem}[ntn]{Remark}
\newtheorem{exa}[ntn]{Example}

\numberwithin{equation}{section}

\newcommand{\N}{\mathbb{N}}
\newcommand{\Z}{\mathbb{Z}}
\newcommand{\q}{\mathbb{Q}}
\newcommand{\R}{\mathbb{R}}

\newcommand{\A}{\mathcal{A}}

\newcommand{\E}{\mathcal{E}}

\newcommand{\mt}{\mapsto}
\newcommand{\arr}{\rightarrow}
\newcommand{\larr}{\longrightarrow}

\newcommand{\se}{\subseteq}
\newcommand{\two}{\twoheadrightarrow}
\newcommand{\tail}{\rightarrowtail}

\newcommand{\vb}{{\rm vb}}

\newcommand{\Hom}{{\rm Hom}}

\newcommand{\corr}{{\rm cor}}
\newcommand{\ress}{{\rm res}}
\newcommand{\id}{{\rm id}}
\newcommand{\tor}{{{\rm Tor}}}
\newcommand{\FP}{{\rm FP}}
\newcommand{\GL}{{\rm GL}}
\newcommand{\supdim}{{\sup_{m\geq 1}\dim_\q}}

\newtheoremstyle{athm}
  {}
  {}
  {\itshape}
  {}
  {\scshape}
  {}
  {.5em}
  {\thmnote{#3}}
\theoremstyle{athm}
\newtheorem*{athm}{}

\begin{document}

\title[Virtual rational Betti numbers]
{Virtual rational Betti numbers of nilpotent-by-abelian groups}
\author{Behrooz  Mirzaii}
\author{Fatemeh Y. Mokari}
\begin{abstract}
In this paper we study virtual rational Betti numbers of a nilpotent-by-abelian group $G$,
where the abelianization $N/N'$ of its nilpotent part $N$ satisfies certain tameness 
property. More precisely, we prove that if $N/N'$ is $2(c(n-1)-1)$-tame 
as a $G/N$-module, $c$ the nilpotency class of $N$, then 
\[
\vb_j(G):=\sup_{M\in\A_G}\dim_\q H_j(M,\q)
\]
is finite for all $0\leq j\leq n$, where $\A_G$ is the set of all finite index 
subgroups of $G$.
\end{abstract}
\maketitle

\section*{Introduction}

The virtual rational Betti numbers of a finitely generated group studies the growth 
of the Betti numbers of the group as one follows passage to subgroups of finite index.
Following \cite{bridson-kochloukova2013} and \cite{kochloukova-mokari2014}, we define
the $n$-th virtual rational Betti number of a finitely generated group $G$ as
\[
\vb_n(G):=\sup_{M\in \A_G}\dim_\q H_n(M,\q),
\]
where $\A_G$ is the set of all subgroups of finite index in $G$.

In \cite{bridson-kochloukova2013} Bridson and Kochloukova introduced and studied the first
virtual rational Betti number of a finitely generated group $G$ and showed that if $G$ 
is either a finitely presented nilpotent-by-abelian group or an abelian-by-polycyclic group 
of type $\FP_3$, then $\vb_1(G)$ is finite. Moreover, they conjectured that this should be
true for all finitely presented soluble groups. As they have shown the finiteness of the
first virtual rational Betti numbers of a metabelian group $G$, with normal abelian subgroup $A$ 
and abelian quotient $Q$ is closely related to the 2-tameness of $A$ as a
$Q$-module, an invariant of metabelian groups introduced by Bieri and Strebel 
\cite{bieri-strebel1980}. 

In \cite{kochloukova-mokari2014}, Kochloukova and the second author extended these results 
to higher virtual rational 
Betti numbers of abelian-by-polycyclic groups, by replacing higher tameness with 
finitely generatedness of high tensor powers of abelian normal subgroups.
More precisely, let $A$ be a normal abelian subgroup of $G$ such that the quotient 
group $Q:=G/A$ is polycyclic. If $Q$ is not abelian, we assume that $G$ is of type 
$\FP_3$. Then it is shown in \cite[Theorem~A]{kochloukova-mokari2014} 
that if $\bigotimes_\q^{2n}(A\otimes_\Z\q)$ is finitely generated as a
$\q Q$-module via the diagonal action, then $\vb_j(G)$ is finite
for $0 \leq j \leq n$. Note that if $G$ is metabelian, then finitely generatedness
of $\bigotimes_\q^{2n}(A\otimes_\Z\q)$ is equivalent to 2n-tameness of $A$ as a
$Q$-module (see Theorem~\ref{b-g}).

Finitely generated soluble groups occurring in applications are often 
nilpotent -by-abelian-by-finite, that is, any such group $G$ contains subgroups 
$N\unlhd H\unlhd G$ such that $N$ is nilpotent, $H/N$ abelian and $G/H$ finite. 
In this paper, we study the virtual rational Betti numbers of nilpotent-by-abelian-by-finite 
groups. Since $\vb_n(G)=\vb_n(H)$ (Lemma~\ref{5}), it is sufficient 
to study virtual rational Betti numbers of nilpotent-by-abelian groups. 
Here is our main theorem.

\begin{athm}[{\bf Tehorem \ref{d-f4}.}]
Let $N \tail G \two Q$ be an exact sequence of groups, where $G$ is finitely generated, 
$N$ is nilpotent of class $c$ and $Q$ is abelian. If $N/N'$ is $2(c(n-1)+1)$-tame, 
then for any $ 0 \leq j \leq n$, $\vb_j(G)$ is finite.
\end{athm}

As a motivation for the study of virtual rational Betti numbers,
one can mention a result of L\"uck which says that the $L_2$-Betti numbers can be computed as
a limit involving the ordinary Betti numbers of subgroups of finite index. 
Here we show that for these groups there is no growth, i.e. the sequences
remain bounded. This result therefore confirms L\"uck's formula by
establishing a stronger property for this class of groups \cite{luck1994}.

To prove our main theorem we needed to study certain aspects of homology of nilpotent groups.
Nilpotent groups have a great deal of commutativity built into their structure and
they are groups that are ``almost abelian". So it is natural to expect that some of 
the properties of homology of abelian groups, in some way, may be shared by nilpotent groups.
In this article, we will study two such properties.
For more similarity between homology of
abelian and nilpotent groups we refer the interested reader to 
\cite{dwyer1975}, \cite{robinson1976}, \cite{h-m-r1975}.

The $n$-th homology of an abelian group $A$ with rational coefficients is isomorphic to
$\bigwedge^{n}_\q (A\otimes_\Z \q)$. We prove the analogue of this result for nilpotent groups.
More precisely, if $N$ is a nilpotent group of class $c$, then we show that
there exists a natural filtration 
of $H_j(N,\q)$,
\[
0=E_0 \se E_1 \se \cdots \se E_{l-1} \se E_l= H_j(N,\q),
\]
such that for any $0\leq k \leq l,$ $E_k/E_{k-1}$ is a natural subquotient 
of a vector space from the set 
$\{\bigotimes^s_\q V\}_{0 \leq s \leq c(j-1)+1}$, where $V:=(N/N')\otimes_\Z \q$.
When our group is free nilpotent, we show that the above theorem is true even with integral 
coefficients. Although the existence of the above filtration is not a surprise and can be 
obtain by easy induction, but the bound $c(j-1)+1$ is new and important for our applications. 
Furthermore, for groups with small $c$ we show that this bound is sharp. The proofs of these
results occupy Sections \ref{section-diff} and \ref{section-nilp}.

Let $N$ be a nilpotent normal subgroup of a group $G$. If $G$ acts nilpotently
on $N/N'$, then Theorem \ref{nil-filtration} implies that $G$ acts nilpotently on 
$H_k(N, \q)$. But with a direct method we can prove a more general result. Let $T$ be an 
$RG$-module, where $R$ is a commutative ring. In Section~\ref{section-nil-action},
we will show that if $G$ acts nilpotently on both $N/N'$ and $T$, 
then $G$ acts nilpotently on each $H_k(N, T)$ and $H^k(N, T)$. 
As an application, 
we show that if moreover $G/N$ is finite and $l$-torsion and  $1/l \in R$, then the natural 
action of $G/N$ on $H_k(N, T)$ and $H^k(N, T)$ is trivial and therefore the natural maps
\[
\corr_N^G:H_k(N,T)\arr H_k(G, T), \ \ \ress_N^G:H^k(G,T)\arr H^k(N,T)
\]
are isomorphisms. 

Both of these results about the homology of nilpotent groups are used in the 
proof of our main theorem (Theorem \ref{d-f4}).

\medskip
{\bf Acknowledgments.} 
We would like to thank Prof. D. H. Kochloukova for introducing the problem to us and 
for her constructive suggestion during the preparation of this paper. Example \ref{koch}
was suggested by her. Moreover, the special case of Theorem \ref{nil-filtration} for
$c=2$ also was proved by J.~R.~Groves which was made available to us by D.~H.~Kochloukova.
His proves is different than ours. We would like to thank them for their helps and suggestions.
The second author is supported by Capes/CNPq PhD grant.

\section{Differentials of the lyndon-hochschild-serre spectral sequence}\label{section-diff}

Let $G$ be a group, $A$ an abelian normal subgroup of $G$ and $Q:=G/A$. Let 
\[
{}_M\E^2_{p,q}=H_p(Q ,H_q(A, M))\Rightarrow H_{p+q}(G,M)
\]
be the Lyndon-Hochschild-Serre spectral sequence associated
to the exact sequence of groups 
\[
A \tail G \two Q,
\]
where here $M$ is either $\Z$ or $\q$ with the trivial action of $G$. In this 
section, we would like to give an explicit formula for the differentials
\[
d^2_{2,q}:{}_\q \E_{2,q}^2\arr {}_\q \E_{0,q+1}^2, 
\]
for any $q\geq 0$, when $A$ is central, i.e. $A \se Z(G)$.

Let $\phi: A \otimes_\Z H_{q}(A, M) \arr H_{q+1}(A, M)$
be the natural product map \cite[Chap.~V, \S5]{brown1994}, say induced by
the shuffle product on the bar resolution, and consider the following composition 
\begin{align}\label{cup-prod}
H^2(Q, A) \otimes_\Z H_p(Q, H_q(A,M)) & 
\overset{- \cap -}{-\!\!\!-\!\!\!-\!\!\!-\!\!\!\larr} 
H_{p-2}(Q, A\otimes _\Z H_q(A,M))
\end{align}
\[
\hspace{6 cm}
\overset{H_{p-2}(\id_Q,\phi)}{-\!\!\!-\!\!\!-\!\!\!-\!\!\!-\!\!\!-\!\!\!-\!\!\!\larr}
H_{p-2}(Q, H_{q+1}(A,M)),
\]
where $- \cap -$ is the cap product \cite[Chap.~V, \S3]{brown1994}. 

Let $\rho$ be the element of $H^2(Q, A)$ associated to 
$A \tail G \two Q$ \cite[Chap.~IV, Theorem~3.12]{brown1994} and set
\[
\Delta(\rho):=H_{p-2}(\id_Q, \phi)\circ (\rho \cap -): 
H_p(Q, H_q(A, M)) \arr H_{p-2}(Q, H_{q+1}(A, M)).
\]
Andr\'e has proved the following fact.

\begin{prp}\label{andre}
Let an exact sequence $A \tail G \two Q$ be given as in above. Then
\[
d_{p,q}^2=d'^2_{p,q} + \Delta(\rho),
\]
where $d'^2_{p,q}$ is the differential of the Lyndon-Hochschild-Serre spectral sequence 
associated to the semidirect product extension 
$A \tail A \rtimes Q \two Q$. 
\end{prp}
\begin{proof}
See \cite[p. 2670]{andre1965}
\end{proof}

Now let $A$ be a central subgroup of $G$.
Then the conjugate action of $Q$ on $A$ is trivial and thus 
$A \rtimes Q=A \times Q$. It is well-known and easy to prove 
that in this case, for any $p$ and $q$, $d'^2_{p,q}=0$ and therefore 
\begin{align}\label{d-rho}
d_{p,q}^2=\Delta(\rho).
\end{align}

Moreover, since $A$ is central, the action of $Q$ on $H_q(A,M)$ is trivial. 
Thus for $M=\q$, the Universal Coefficient Theorem implies that
\[
\begin{array}{c}
{}_\q \E^2_{p,q}=H_p(Q,\Z)\otimes_\Z H_q(A,\q)\simeq
H_p(Q,\Z)\otimes_\Z \bigwedge_\q^q (A \otimes_\Z \q).
\end{array}
\]  
If $p=2$, then (\ref{cup-prod}) finds the following form
\[
H^2(Q, A) \otimes_\Z H_2(Q,\Z) \otimes_\Z H_q(A,\q)  
\overset{(- \cap -)\otimes \id}{-\!\!\!-\!\!\!-\!\!\!-\!\!\!\larr} 
A\otimes _\Z H_q(A,\q) \overset{\phi}{\arr} H_{q+1}(A,\q),
\]
where \[
-\cap -: H^2(Q, A) \otimes_\Z H_2(Q,\Z) \arr A
\]
is the cap product. Therefore from formula (\ref{d-rho}), we obtain the 
following explicit formula
\[
\begin{array}{c}
d_{2,q}^2: {}_\q\E_{2,q}^2=H_2(Q,\Z)\otimes_\Z \bigwedge_\q^q (A \otimes_\Z \q) \arr 
{}_\q\E_{0,q+1}^2=\bigwedge_\q^{q+1} (A \otimes_\Z \q),
\end{array}
\]
\[
x \otimes (a_1\wedge \dots\wedge a_{q})
\mapsto (\rho \cap x)\wedge a_1\wedge \dots \wedge a_{q}.
\]
Thus we have proved  the following proposition.

\begin{prp}\label{d-spectral}
Let $G$ be a group, $A$ a central  subgroup of $G$ and $Q:=G/A$. Let 
\[
{}_\q\E^2_{p,q}=H_p(Q ,H_q(A ,\q ))\Rightarrow H_{p+q}(G,\q)
\]
be the Lyndon-Hochschild-Serre spectral sequence associated
to the extension $A \tail G \two Q$. Then for any $q\geq 0$, the differential 
\[
\begin{array}{c}
d_{2,q}^2: {}_\q\E_{2,q}^2=H_2(Q,\Z)\otimes_\Z \bigwedge_\q^q (A \otimes_\Z \q) \arr 
{}_\q\E_{0,q+1}^2=\bigwedge_\q^{q+1} (A \otimes_\Z \q),
\end{array}
\]
is given by the formula
$x \otimes (a_1\wedge \dots\wedge a_{q})
\mapsto (\rho \cap x)\wedge a_1\wedge \dots \wedge a_{q}$. Here $\rho$ is the element of $H^2(G,A)$
associated to the above extension and the map $-\cap -: H^2(Q, A) \otimes_\Z H_2(Q,\Z) \arr A$ is the
cap product. If $A$ is torsion free, then the same result is true for 
$d_{2,q}^2: {}_\Z\E^2_{2,q} \arr {}_\Z\E^2_{0,q+1}$.
\end{prp}

The following corollary will be needed in the next section.

\begin{cor}\label{d-surjective}
Let $G$, $A$, $Q$ and ${}_\q\E^2_{p,q}$ be as in Proposition \ref{d-spectral}.
If $A \se Z(G)\cap G'$, then $d_{2,q}^2: {}_\q\E^2_{2,q} \arr {}_\q\E^2_{0,q+1}$ 
is surjective for any $q\geq 0$ and therefore ${}_\q\E^\infty_{0,q}= {}_\q\E^3_{0,q}=0$. 
Moreover, if $A$ is torsion free, then the same results hold for
$d_{2,q}^2: {}_\Z\E^2_{2,q} \arr {}_\Z\E^2_{0,q+1}$.
\end{cor}

\begin{proof}
The spectral sequence ${}_M\E^2_{p,q}$,
gives us the five term exact sequence
\[
H_2(G,M)\arr H_2(Q,M)\overset{d_{2,0}^2}{\larr} H_1(A,M)_{Q}\arr 
H_1(G,M)\arr H_1(Q,M)\arr 0,
\]
\cite[Chap.~VII, Corollary~6.4]{brown1994}. Clearly $H_1(G,\Z)\simeq H_1(Q,\Z)\simeq G/G'$.
Since the action of $Q$ on $A$ is trivial, we have 
$H_1(A,\Z)_{Q}\simeq H_1(A,\Z)=A$. Thus from the above exact sequence,
we obtain the surjective map 
\[
d_{2,0}^2:H_2(Q,\Z)\two A.
\]
But, from the above, we know that this map is given by the formula $x \mapsto \rho \cap x$.
Now by Proposition \ref{d-spectral}, $d_{2,q}^2$ is surjective and
this immediately implies that $\E^\infty_{0,q}= \E^3_{0,q}=0$.
\end{proof}

\section{Homology of nilpotent groups}\label{section-nilp}

Let $N$ be a nilpotent group of class $c$ and consider its lower central series,
\[
1=\gamma_{c+1}(N)\subset \gamma_{c}(N)\subset \cdots \subset 
\gamma_{2}(N)\subset \gamma_{1}(N)=N.
\]
From the exact sequence $ \gamma_{c}(N) \tail N \two N/\gamma_{c}(N)$, we obtain
the Lyndon-Hochschild-Serre spectral sequence 
\begin{align}\label{spectral}
E^2_{p,q}=H_p(N/\gamma_{c}(N) ,H_q(\gamma_{c}(N) , T )) \Rightarrow H_{p+q}(N, T),
\end{align}
where $T$ is a $N$-module.

Since $\gamma_{c+1}(N)=[\gamma_{c}(N),N]=1$, it follows that $\gamma_{c}(N)\se Z(N)$.
So the conjugate action of $N/\gamma_{c}(N)$ on $\gamma_{c}(N)$ is trivial. 
This also implies that
the action of $N/\gamma_{c}(N)$ on $H_q(\gamma_{c}(N) ,T )$ is trivial, provided that
the action of $N$ on $T$ is trivial.

\begin{thm}\label{nil-filtration}
Let $N$ be a nilpotent group of class $c$. Then there exists a natural 
filtration of $H_j(N,\q)$,
\[
0=E_0 \se E_1 \se \cdots \se E_{l-1} \se E_l= H_j(N,\q),
\]
such that for any $0\leq k \leq l,$ $E_k/E_{k-1}$ is a natural 
subquotient of a vector space from the set 
$\{\bigotimes^s_\q V\}_{0 \leq s \leq c(j-1)+1}$, where $V:=(N/N')\otimes_\Z \q$.
\end{thm}

\begin{proof}
We prove the claim by induction on $c$. All filtrations, 
homomorphisms and subquotients that will be considered in this proof are natural. 
If $c=1$, then $N'=\gamma_{2}(N)=1$. Thus $N$ is abelian and 
by \cite[Theorem 6.4, Chap. V]{brown1994} we have
\[
\begin{array}{c}
H_j(N,\q)\simeq (\bigwedge^j_\Z N) \otimes_\Z \q \simeq \bigwedge^j_\q V.
\end{array}
\]
Clearly $\bigwedge^j_\q V$ is of the form $(\bigotimes^j_\q V)/T$, for some subspace $T$ of
$\bigotimes^j_\q V$. Since $j=1(j-1) +1=c(j-1)+1$, our claim is valid for $c=1$. 

Now let $c\geq 2$ and assume that the claim of the theorem is true for all nilpotent groups of 
class $d$, $1 \leq d \leq c-1$. 
The spectral sequence (\ref{spectral}) gives us a filtration of $H_j(N,\q)$
\[
0=F_{-1}H_j \se F_0H_j \se \cdots \se F_{j-1}H_j \se F_jH_j= H_j(N,\q),
\]
such that $E^\infty_{i,j-i}\simeq F_iH_j/F_{i-1}H_j$, $0\leq i \leq j$. 
By Corollary \ref{d-surjective},
$E^\infty_{0,j}=0$, so $F_0H_j=F_0H_j/F_{-1}H_j\simeq E^\infty_{0,j}=0$.  

We know that $E^\infty_{i,j-i}$ is a subquotient of
\[
E^2_{i,j-i} \simeq H_i(N /
\gamma_{c}(N),\q)\otimes_\q H_{j-i}(\gamma_{c}(N),\q ).
\]
The group $\gamma_{c}(N)$ is abelian, so 
\[
\begin{array}{c}
H_{j-i}(\gamma_{c}(N),\q)\simeq \bigwedge^{j-i}_\q ( \gamma_{c}(N)\otimes_\Z \q ).
\end{array}
\]
There is a natural surjective map $\bigotimes^{c}_\Z (N/N')\two \gamma_{c}(N)$,
which induces a surjective map
\[
\begin{array}{c}
\bigwedge^{j-i}_\q (\bigotimes_\q^c V) \two \bigwedge^{j-i}_\q ( \gamma_{c}(N)\otimes_\Z \q )
\end{array}
\]
and clearly from this we obtain a surjective map
\begin{align}
\begin{array}{c}
\bigotimes_\q^{c(j-i)} V \two H_{j-i}(\gamma_{c}(N),\q).
\end{array}
\end{align}
This implies that $F_iH_j/F_{i-1}H_j$ is a subquotient of
\begin{equation} \label{1}
\begin{array}{c}
H_i(N /\gamma_{c}(N),\q) \otimes_\q \bigotimes_\q^{c(j-i)} V.
\end{array}
\end{equation}
 On the other hand, since $N/\gamma_{c}(N)$ is nilpotent
of class $c-1$, by the induction hypothesis, for any $1\leq i \leq j$, we have a filtration of
$H_{i}(N/\gamma_{c}(N),\q)$,
\[
0=G_{0,i} \se G_{1,i} \se \cdots \se G_{k_i-1,i} \se G_{k_i,i}= H_{i}(N/ \gamma_{c}(N),\q)
\]
such that for any $0\leq t \leq k_i$, $G_{t,i}/G_{t-1,i}$ is a subquotient of some 
$\bigotimes_\q^{s_{t,i}} V$, where $0\leq s_{t,i}\leq (c-1)(i-1)+1$.
(Note that $(N/\gamma_{c}(N))/(N/\gamma_{c}(N))'=N/N')$.
This together with (\ref{1}) imply that $F_iH_j/F_{i-1}H_j$ is a subquotient of some
$\bigotimes_\q^{s_{i}} V$, where 
\[
0\leq s_{i}\leq (c-1)(i-1)+1 + c(j-i) = c(j-1) - i + 2 \leq c(j-1) + 1.
\]
This finishes the induction step and so the proof of the theorem. 
\end{proof}

With some restriction on $N$, one can obtain similar results for integral homology.

\begin{prp}\label{nil-filtration1}
Let $N$ be a free nilpotent group of class $c$. 
Then there exists a natural filtration of $H_j(N,\Z)$,
\[
0=E_0 \se E_1 \se \cdots \se E_{l-1} \se E_l= H_j(N,\Z),
\]
such that for any $0\leq k \leq l,$ $E_k/E_{k-1}$ is a natural subquotient of 
a $\Z$-module from the set 
$\{\bigotimes^s_\q V\}_{0 \leq s \leq c(j-1)+1}$, where $V:=N/N'$. 
\end{prp}
\begin{proof}
Since $N$ is a free nilpotent group, $\gamma_c(N)$ is torsion free. Thus
\[
\begin{array}{c}
H_n(\gamma_c(N),\Z)\simeq \bigwedge^n_\Z \gamma_c(N)
\end{array}
\]
(see \cite[Theorem 6.4, Chap. V]{brown1994}) and so it is torsion free.
This implies that
\[
E^2_{i,j-i} \simeq H_i(N /
\gamma_{c}(N),\Z)\otimes_\Z H_{j-i}(\gamma_{c}(N),\Z ).
\]
Now the proof is similar to the proof of Theorem \ref{nil-filtration}. 
\end{proof}

\begin{rem}
We believe that $c(j-1)+1$ is a sharp bound for the existence of
a filtration with the above property for $H_j(N, \q)$. 
At least this is true for the extreme cases $c=1$ (abelian $N$) 
or $j=1$ 
(first homology group case). Also the above proof shows that
$E_1=F_1H_j$ is a quotient of $\bigotimes_\Z^{c(j-1)+1}V$.
This gives an evidence for the fact that the bound $c(j-1)+1$ 
in Theorem \ref{nil-filtration} is sharp.
\end{rem}

\begin{rem}
If $N$ is a nilpotent group of class $c$, then the 
above theorem also is true for $H_2(N,\Z)$. By this we mean 
that there exist a natural filtration of $H_2(N,\Z)$,
\[
0=E_0 \se E_1 \se \cdots \se E_{l-1} \se E_l= H_2(N,\Z),
\]
such that for any $0\leq k \leq l, $ $E_k/E_{k-1}$ is a natural subquotient 
of a $\Z$-module from the set $\{\bigotimes^s_\Z(N/N')\}_{0 \leq s \leq c+1}$. 
This follows from the above proof, using the facts that for an abelian 
group $A$, $H_2(A, \Z)\simeq A\wedge A$ and also for $0 \leq i \leq 2$,
\[
E^2_{i,2-i}\simeq H_i(N/\gamma_{c}(N),\Z)\otimes_\Z H_{2-i}(\gamma_{c}(N) ,\Z ).
\]
If $c=2$, the complete structure of $H_2(N,\Z)$ is established 
in \cite{desiPhD}. This description is simple if $N$ is torsion-free.  
In this case $N / \gamma_2(N)$ is torsion-free and we obtain a filtration
\[
0 \se F_1H_2 \se F_2H_2=H_2(N,\Z)
\]
such that 
\[
F_1H_2\simeq \frac{(N/N') \otimes_\Z N'}{\langle 
xN'\otimes [y,z]+yN'\otimes [z,x]+zN'\otimes [x,y]\mid x,y,z \in N \rangle}
\]
and
\[
F_2H_2/F_1H_2\simeq \ker\bigg((N/N')\wedge (N/N') \larr N', xN'\wedge yN' 
\mapsto [x,y] \bigg).
\]
\end{rem}

\begin{rem}
Let $N$ be a free nilpotent group of finite rank and of class 
$c=2$. Then by \cite[p. 532]{kuzmin-semenov1998}, the differential 
\[
\begin{array}{c}
d_{p,q}^2: E_{p,q}^2=\bigwedge^p_\Z(N/N')\otimes_\Z\bigwedge^q_\Z N'\!\arr\!
E_{p-2,q+1}^2=\bigwedge^{p-2}_\Z(N/N') \otimes_\Z\bigwedge^{q+1}_\Z N'
\end{array}
\]
of the spectral sequence (\ref{spectral}) is given by the formula
\[
\hspace{-5.5 cm}
d_{p,q}^2(a_1N'\wedge \dots \wedge a_pN' \otimes x_1\wedge \dots 
\wedge x_q)=
\]
\[
\sum_{k<l}\!(-1)^{k+l-1}
a_1N'\!\wedge \dots \widehat{a_kN'} \dots \widehat{a_lN'} \dots \wedge a_pN'
\otimes [a_k,a_l]\wedge x_1\wedge \dots \wedge x_q.
\]
Also in \cite[Theorem 4]{kuzmin-semenov1998}, it is shown that
\[
H_j(N, \Z)\simeq \bigoplus_{i=1}^j E_{i,j-i}^3
\]
(note that $E_{0,j}^3=0$). This means that the filtration 
of $H_j(N,\Z)$ induced by the spectral sequence,
\[
0=F_0H_j \se F_1H_j \se \cdots \se F_{j-1}H_j \se F_jH_j= H_j(N,\Z),
\]
has the form
\[
\begin{array}{c}
F_iH_j/F_{i-1}H_j\simeq E_{i,j-i}^3\se \bigg(\bigwedge^{i}_\Z(N/N') 
\otimes_\Z\bigwedge^{j-i+1}_\Z N'\bigg)/T_{i,j-i},
\end{array}
\]
where $T_{i,j-i}$ is generated by the elements
\[
\sum_{k<l}\!(-1)^{k+l-1}
y_1\wedge \dots \wedge\widehat{y_k}\wedge \dots \wedge\widehat{y_l}\wedge \dots 
\wedge y_{i+2}\otimes [y_k,y_l]
\wedge x_1\wedge \dots \wedge x_{j-i-1},
\]
where $y_h\in N/N'$, $x_g \in N'$.
This shows that $F_1H_j\simeq E_{1,j-1}^3$ from the filtration 
is a quotient of $\bigotimes_\Z^{2j-1}(N/N')$ and is 
non-trivial. So the bound $2j-1=c(j-1)+1$ in Theorem 
\ref{nil-filtration} is sharp.
\end{rem}

\begin{cor}\label{nil-filtration-2}
Let $N \tail G \two Q$ be an exact sequence of groups, where  
$N$ is nilpotent of class $c$. Then 
there exist a natural filtration of $\q Q$-submodules 
of $H_j(N,\q)$,
\[
0=E_0 \se E_1 \se \cdots \se E_{l-1} \se E_l= H_j(N,\q),
\]
such that for any $0\leq k \leq l,$ $E_k/E_{k-1}$ is a natural subquotient 
of a $\q Q$-module from the set 
$\{\bigotimes^s_\q V\}_{0 \leq s \leq c(j-1)+1}$, 
where $V:=(N/N')\otimes_\Z \q$, and $\bigotimes^s_\q V$
is considered as a $\q Q$-module via the diagonal action of $Q$.
\end{cor}

\begin{proof}
We have a natural action of $Q$ on
$H_q(\gamma_c(N),\q)$ and $H_p(N/\gamma_c(N),\q)$.
From these we obtain a natural action of $Q$ on the 
Lyndon-Hochschild-Serre spectral sequence
\[
E^2_{p,q}=H_p(N/\gamma_{c}(N) ,H_q(\gamma_{c}(N) ,\q )) 
\Rightarrow H_{p+q}(N,\q).
\]
This means that the groups $E_{p,q}^2$ are $\q Q$-modules
and the differentials $d_{p,q}^2$ are homomorphisms of 
$\q Q$-modules. This implies that we have a filtration 
of $\q Q$-submodules of $H_j(N,\q)$
\[
0=F_{-1}H_j \se F_0H_j \se \cdots \se F_{j-1}H_j \se F_jH_j= H_j(N,\q),
\]
such that each $E^\infty_{i,j-i}\simeq F_iH_j/F_{i-1}H_j$, 
$0\leq i \leq j$, is an isomorphism of $\q Q$-modules.

It is also easy to see that if $\bigotimes_\Z^c(N/N')$ 
is considered as $\Z Q$-module via the diagonal action 
of $Q$, then the natural map $\bigotimes_\Z^c(N/N')\arr \gamma_c(N)$ 
is a homomorphism of $\Z Q$-modules. 
Now if we follow the proof of Theorem~\ref{nil-filtration}, 
we see that in all steps of the
proof the $\q Q$-structure is preserved. This means that 
all subquotients considered in the proof of Theorem \ref{nil-filtration} 
are $\q Q$-subquotients (i.e. the subquotient structure commutes with the $Q$-action) 
and the maps are $\q Q$-homomorphisms, etc. Therefore, as in the proof of
Theorem \ref{nil-filtration}, we obtain the desired 
filtration.
\end{proof}

\section{Nilpotent action on the homology of nilpotent groups}\label{section-nil-action}

We say that a group $G$ acts nilpotently on a $G$-module $T$, if $T$ has a filtration 
of $G$-submodules 
\[
0=T_0 \se T_1 \se \cdots \se T_{k-1} \se T_k=T,
\]
such that the action of $G$ on each quotient $T_i/T_{i-1}$ is trivial. 

Corollary \ref{nil-filtration-2} shows that if $Q=G/N$ acts nilpotently on $N/N'$, then 
it act nilpotently on $H_j(N, \q)$ for any $j\geq 0$. This fact can be generalized
as follow.

\begin{thm}\label{nilpotent-action}
Let $G$ be a group, $N$ a nilpotent normal subgroup of $G$ and let $T$ be a $G$-module.  
If $G$ acts nilpotently on $N/N'$ and $T$, then, for any $k\geq 0$, $G$ acts nilpotently on 
$H_k(N, T)$ and $H^k(N, T)$.
\end{thm}
\begin{proof}
We prove the claim for the homology functor. The proof for the cohomology functor is similar.
The proof is in three steps.
\medskip

{\bf Step 1.} {\it $N$ is abelian and $T$ is a trivial $G$-module}: Let
\[
0=N_0 \se N_1 \se \cdots \se N_n=N
\]
be a filtration of $N$ such that $G$ acts trivially
on each quotient $N_i/N_{i-1}$. We prove this step by induction on the length of the filtration of $N$,
i.e. on $n$. If $n=1$, then the action of $G$ on $N=N_1$ is trivial. So the action of $G$
on $H_k(N, T)$ also is trivial. From the exact sequence of groups
\[
N_1 \tail N \two N/N_1
\] 
we obtain the Lyndon-Hochschild-Serre spectral sequence
\[
{E'}^2_{p,q}=H_p(N/N_1, H_q(N_1, T )) \Rightarrow  H_{p+q}(N, T).
\]
By above, $G$ acts trivially (and so nilpotently) on $H_q(N_1, T)$. Since $G/N_1$ 
acts nilpotently on $N/N_1$ and $N/N_1$ has
a filtration of length $n-1$, by induction hypothesis
$G/N_1$, and so $G$, acts nilpotently on each ${E'}^2_{p,q}$.
Since ${E'}^\infty_{p,q}$ is a subquotient of ${E'}^2_{p,q}$, $G$ acts nilpotently on it too.
Moreover, $G$ acts naturally on the above spectral sequence which means that each
${E'}_{p,q}^2$ is a $G$-module and the differentials ${d'}_{p,q}^2$ are homomorphisms of 
$G$-modules. This implies that we have a filtration  of $G$-submodules 
\[
0=F_{-1}H_k \se F_0H_k \se \cdots \se F_{k-1}H_k \se F_kH_k= H_k(N,T),
\]
such that each isomorphism ${E'}^\infty_{i,k-i}\simeq F_iH_k/F_{i-1}H_k$ is an isomorphism of 
$G$-modules. Thus $G$ acts nilpotently on each quotient $F_iH_k/F_{i-1}H_k$. This implies that 
$G$ acts nilpotently on $H_k(N, T)$.
\medskip

{\bf Step 2.} {\it $N$ is abelian and $T$ is any $G$-module}: Let 
\[
0=T_0 \se T_1 \se \cdots \se T_l=T
\]
 be a filtration
of $T$, such that $G$ acts trivially on each quotient $T_i/T_{i-1}$. In this case
we prove the theorem by induction on $l$, the length of the filtration of $T$. If $l=1$, then the 
action of $G$ on $T=T_1$ is trivial, so we arrive at Step~1. From the exact sequence 
\[
0 \arr T_1 \arr T \arr T/T_1 \arr 0
\]
we obtain the long exact sequence
\[
\cdots \arr H_k(N, T_1) \arr H_k(N, T) \arr H_k(N, T/T_1) \arr \cdots.
\]
We know that $G$ acts nilpotently on $H_k(N, T_1)$ and by the induction hypothesis $G$ 
acts nilpotently on $H_k(N, T/T_1)$. Now the above exact sequence implies that
$G$ acts nilpotently on $H_k(N, T)$.
\medskip

{\bf Step 3.} {\it The general case}:
The proof of this step is by induction on the nilpotent class $c$ of $N$. If $c=1$, 
then $N$ is abelian and this is done in Step~2. Now assume that the claim is true 
for all nilpotent groups of class $d$,  $1\leq d \leq c-1$. Consider the lower central 
series of $N$,
\[
1=\gamma_{c+1}(N)\subset \gamma_{c}(N)\subset \cdots \subset 
\gamma_{2}(N)\subset \gamma_{1}(N)=N.
\]
Note that $\gamma_{c}(N)\se Z(N)$. The exact sequence of groups
\[
\gamma_{c}(N)  \tail   N  \two   N/ \gamma_{c}(N),
\]
gives us the Lyndon-Hochschild-Serre spectral sequence
\[
E^2_{p,q}=H_p(N/\gamma_{c}(N), H_q(\gamma_{c}(N) , T )) \Rightarrow  H_{p+q}(N, T).
\]
We have a natural surjective map 
\[
\begin{array}{c}
\bigotimes_\Z^c(N/N') \two \gamma_{c}(N)
\end{array}
\] 
which is a map of $G$-modules if we consider $\bigotimes_\Z^c (N/N')$ as a $G$-module
via the diagonal action \cite[1.2.11]{lennox-robinson2004}. Since $G$ acts nilpotently on 
$N/N'$, it also acts nilpotently on $\bigotimes_\Z^c (N/N')$. Thus through the above 
surjective map, $G$ also acts nilpotently on $\gamma_{c}(N)$. 
By Step 2, $G$ acts nilpotently on $H_q(\gamma_{c}(N), T)$. 
On the other hand, $N/\gamma_{c}(N)$ is of nilpotent class $c-1$ and $G$ acts nilpotently on 
$(N/\gamma_{c}(N))/(N/\gamma_{c}(N))'\simeq N/N'$. So by the induction hypothesis, $G$ acts
nilpotently on each $E_{p,q}^2$. Thus $G$ acts 
nilpotently on each $E^\infty_{p,q}$. Finally by the convergence of the spectral sequence, 
one can show, as in Step 1, that $G$ acts nilpotently on $H_k(N, T)$. This completes the 
proof of the theorem.
\end{proof}

If $A$ is an abelian normal subgroup of $G$, then one can show that $G$ is nilpotent 
if and only if $G/A$ is nilpotent and $G$ acts nilpotently on $A$ \cite[Proposition~4.1, 
Chap. I]{h-m-r1975}. One side of this fact can be generalized as follow.

\begin{cor}\label{nilpotent-action2}
Let $G$ be a nilpotent group, $N$ a normal subgroup of $G$ and let $T$ be a $G$-module. 
If $G$ acts nilpotently on $T$, then for any $k\geq 0$, $G/N$ acts nilpotently on 
$H_k(N, T)$ and $H^k(N, T)$.
\end{cor}
\begin{proof}
Since $G/N'$ is nilpotent and $N/N'$ is abelian, $G/N'$,
and so $G$, acts nilpotently on $N/N'$. Now the claim follows from Theorem \ref{nilpotent-action}. 
\end{proof}

\begin{lem}\label{trivial-action}
Let $G$ be a finite group, $R$ a commutative ring and $T$ an $RG$-module such that
$G$ acts nilpotently.
\par {\rm (i)} If $1/|G| \in R$, then $T$ is a trivial $G$-module.
\par {\rm (ii)} If $G$ is nilpotent, $l$-torsion and  $1/l \in R$, then $T$
is a trivial $G$-module.
\end{lem}
\begin{proof}
(i) We know that the functor $-\otimes_G\Z=(-)_G$ is right exact. First we show that this is 
in fact an exact functor if it is considered as a functor from the category of $RG$-modules 
to the category of $R$-modules. Consider the maps 
\[
\alpha_G: T^G \arr T_G, \ \ \ \ m \mt \overline{m},
\]
and
\[
\overline{N}: T_G \arr T^G,  \ \ \ \ \overline{m} \mt Nm,
\]
where $N:=\sum_{g \in G}g \in R G$. Then clearly $\overline{N}\circ \alpha$ and
$\alpha \circ\overline{N}$ coincide with multiplication by $|G|$. Since $1/|G|\in R$,
$\alpha_G$ is an isomorphism. This implies that $(-)_G$ is exact, because $(-)^G$ is 
left exact. 
Next, let  
\[
0=T_0 \se T_1 \se \cdots \se T_k=T
\]
be a filtration of $T$ such that $G$ acts trivially on each $T_i/T_{i-1}$. 
By applying the exact functor $(-)_{G}$ to the exact sequence 
$0 \arr T_1 \arr T_2 \arr T_2/T_1 \arr 0$ and using the fact that 
$G$ acts trivially on $T_1$ and $T_2/T_1$, we see that 
\[
0 \arr T_1 \arr (T_2)_{G} \arr T_2/T_1 \arr 0
\]
is exact. Therefore $T_2\simeq (T_2)_{G}$ and so the action of $G$ on $T_2$ is trivial. 
In a similar way and by induction on $i$, one can show that the action of $G$ on each 
$T_i$ is trivial. Thus the action of $G$ on $T_k=T$ is trivial. 

(ii) 
First we prove that $(-)_G$ is exact and we do this by induction on the size of $G$. 
We may assume that $G\neq 1$. Since $G$ is nilpotent, 
$Z(G)\neq 1$. Let $H$ be a nontrivial cyclic subgroup of $Z(G)$.
Then the map $\alpha_G$ coincides with the following composition of maps
\[
T^G\overset{\simeq}{\larr} {(T^H)}^{G/H} \overset{\alpha_H}{\larr} {(T_H)}^{G/H}
\overset{\alpha_{G/H}}{\larr} {(T_H)}_{G/H} \overset{\simeq}{\larr} T_G.
\]
Now the exactness of the functor $(-)_G$ follows from (i) and the induction step.
Finally, as in (i) we can prove that $G$ acts trivially on $T$.
\end{proof}

\begin{cor}\label{nil-coristriction}
Let $G$ be a nilpotent group and $N$ a normal subgroup of $G$ such that $G/N$ is finite
and $l$-torsion. Let $R$ be a commutative ring such that $1/l \in R$ and let $T$ be an 
$RG$-module. If $G$ acts nilpotently on $T$, then, for any $k\geq 0$, the natural action of 
$G/N$ on $H_k(N, T)$ and $H^k(N, T)$ is trivial and therefore the natural maps
\[
\corr_N^G:H_k(N,T)\arr H_k(G, T), \ \ \ress_N^G:H^k(G,T)\arr H^k(N,T)
\]
are isomorphisms.
\end{cor}
\begin{proof}
The claim follows from  Corollary \ref{nilpotent-action2} and Lemma 
\ref{trivial-action}.
\end{proof}

\begin{cor}\label{nil-transfer}
Let $G$ be a nilpotent group and $N$ a subgroup of $G$ such that $G/N$ is finite
and $l$-torsion. Let $R$ be a commutative ring such that $1/l! \in R$ and let $T$ be an 
$RG$-module. If $G$ acts nilpotently on $T$, then, for any $k\geq 0$, the natural maps
\[
\corr_N^G:H_k(N,T)\arr H_k(G, T), \ \ \ress_N^G:H^k(G,T)\arr H^k(N,T)
\]
are isomorphisms. 
\end{cor}

\begin{proof}
It is well-known that $N$ has a subgroup $L$ such that $L$ is normal
in $G$ and $[G:L] \leq [G:N]!$. Now by Corollary \ref{nil-coristriction},
the maps 
\[
\corr_L^G:H_k(L,T)\arr H_k(G, T)\ \ \ \text{and} \ \ \ \corr_L^N:H_k(L,T)\arr H_k(N, T)
\]
are isomorphisms. Therefore $\corr_N^G:H_k(N,T)\arr H_k(G, T)$ is an isomorphism.
The cohomology case can be treated in a similar way.
\end{proof}

\begin{exa}
In general, in Corollary \ref{nil-coristriction} the condition that 
$[G:N]< \infty$ and $1/l\in R$ can not be removed. In fact, if  $N$ 
is a non-central abelian normal subgroup of a nilpotent group $G$, 
e.g. $G$ a nilpotent group of class $c=3$ and $N=G'$, then clearly 
$G$ does not act trivially 
on $H_1(N,\Z)=N$.
\end{exa}

\section{Bieri-Strebel invariant}\label{section-tame}

The main condition of our main Theorem \ref{d-f4}, proved below, is closely related to an 
invariant, introduced by Bieri and Strebel \cite{bieri-strebel1980}, which has played a 
prominent role in the study of soluble groups which are finitely presented.
 
Let $Q$ be a multiplicative finitely generated abelian group. A homomorphism of 
groups 
\[
v:Q\arr \R
\]
is called a valuation on $Q$. If $Q$ has rank $n$, then
$\Hom_\Z(Q,\R)\simeq \R^n$, so $\Hom_\Z(Q,\R)$ can be regarded as a topological vector space.
Two valuation $v$ and $v'$ on $Q$ are called equivalent if $v'=a v$ for some
$a\in \R^{>0}$. We denote
the equivalence class of $v$ by $[v]$ and the set $S(Q)$ of all equivalence classes of 
elements of $\Hom_\Z(Q, \R) \setminus \{ 0 \}$ is called the valuation sphere, which 
can be identified with the unit sphere ${\mathbb S}^{n-1}\subset \R^n$.
Notice that $S(Q)$ is empty precisely when 
$n=0$, that is, $Q$ is finite. For any valuation $v$ on $Q$ define
\[
Q_v:=\{q\in Q| v(q)\geq 0\},
\]
which is a submonoid of $Q$. 

For a ring $R$, let $R Q_v$  be the monoid ring, which clearly is a subring of $R Q$.
For a finitely generated $R Q$-module $A$, define
\[
\Sigma_A(Q):=\Big\{[v]\in S(Q) \mid A \ {\rm is \ finitely \ generated \ over} \ R Q_v \Big\}.
\]
A finitely generated $RQ$-module $A$ is called $m$-tame if for any $m$ elements 
$v_1,\dots,v_m\in 
\Hom_\Z(Q,\R)\setminus \{0\}$ with $v_1+\dots+v_m=0$, there is $1\leq i \leq m$
such that $[v_i]\in\Sigma_A(Q)$. 

\begin{thm}\label{b-g}
Let $Q$ be a finitely generated abelian group, $K$ a field, $A$  a finitely generated
$K Q$-module and $m\geq 2$ an integer. Then the following statements are equivalent:
\par {\rm (i)} $A$ is $m$-tame as $K Q$-module,
\par {\rm (ii)} $\bigotimes_K^m A$ is finitely generated as $KQ$-module
via the diagonal $Q$-action,
\par {\rm (iii)} $\bigotimes_K^i A$ are finitely generated as $KQ$-modules
via the diagonal $Q$-action for $i=2, \dots, m$,
\par {\rm (iv)} $\bigwedge_K^i A$ are finitely generated as
$KQ$-modules via the diagonal $Q$-action for $i=2, 3, \dots, m$,
\par {\rm (v)} $\bigwedge_K^m A$ is finitely generated as
$KQ$-module via the diagonal $Q$-action.
\end{thm}
\begin{proof}
See \cite[Theorem~C]{bieri-groves1982} and \cite[Corollary~B]{kochloukova1999}.
\end{proof}

\begin{thm}\label{b-g2}
Let $ A \tail G \two Q$ be a short exact sequence of groups with both $A$ and $Q$
abelian and $G$ finitely generated. If $G$ is of type $\FP_m$, then $A\otimes_\Z K$ is $m$-tame
as a $KQ$-module for every field $K$.
\end{thm}
\begin{proof}
See Theorem D in \cite{bieri-groves1982}.
\end{proof}

\section{Virtual rational Betti numbers of nilpotent-by-abelian groups}\label{section-virtual}

The following two theorems are taken from \cite{bridson-kochloukova2013} and 
\cite{kochloukova-mokari2014}, respectively which are very important for the study 
of virtual rational Betti numbers of abelian-by-polycyclic groups. In this section 
we will use them for the study of virtual rational Betti numbers of nilpotent-by-abelian 
groups.

\begin{thm}[Bridson-Kochloukova]\label{b-k2013}
Let $Q$ be a finitely generated abelian group and $B$ a finitely generated $\q Q$-module. 
If $B \otimes_\q B$ is a finitely generated $\q Q$-module via the diagonal action of $Q$,
then
\[
\sup_{M\in \A_Q}\dim_{\q} (B \otimes_{\q M}\q) < \infty.
\]
\end{thm}
\begin{proof}
See Theorem 3.1 in \cite{bridson-kochloukova2013}.
\end{proof}

\begin{thm}[Kochloukova-Mokari]\label{d-f2}
Let $Q$ be a finitely generated abelian group and $B$ 
a finitely generated $\q Q$-module. If
$\sup_{m \geq 1} \dim_{\q} (B \otimes_{\q Q^m} \q) < \infty$,
then for any $i \geq 0$,
\[
\sup_{m \geq 1} \dim_{\q} H_i(Q^m, B)<\infty.
\]
\end{thm}
\begin{proof}
See Theorem 2.4 in \cite{kochloukova-mokari2014}.
\end{proof}

\begin{lem}\label{subquotient}
Let $Q$ be a finitely generated abelian group. Let $V$ 
be a $\q Q$-module such that 
$\bigotimes_\q^n V$ is a finitely generated $\q Q$-module 
via the diagonal action of $Q$. 
If $\supdim \Big((\bigotimes_\q^n V)\otimes_{\q Q^m}\q\Big) < \infty$,
then for any  $\q Q$-subquotient  $U$ of 
$\bigotimes_\q^n V$, we have
\[
\supdim (U \otimes_{\q Q^m}\q) < \infty.
\]
\end{lem}
\begin{proof}
First let us assume that $U$ is a quotient of $\bigotimes_\q^n V$, i.e. $U=(\bigotimes_\q^n V)/T$, 
for some $\q Q$-submodule $T$ of $\bigotimes_\q^n V$. Then clearly 
\[
\begin{array}{c}
\dim_\q (U \otimes_{\q Q^m}\q) \leq \dim_\q\Big((\bigotimes_\q^n V)\otimes_{\q Q^m}\q\Big)
\end{array}
\]
and thus
\[
\supdim (U \otimes_{\q Q^m}\q) \leq  \supdim 
\begin{array}{c}
\Big((\bigotimes_\q^n V)\otimes_{\q Q^m}\q\Big) < \infty.
\end{array}
\]
Next let $U$ be a $\q Q$-submodule of some $W:=(\bigotimes_\q^n V)/T$. Then $W/U$ is 
of the form $(\bigotimes_\q^n V)/T'$ for some $\q Q$-submodule $T'$ of $\bigotimes_\q^n V$ and so
\[
\supdim (W \otimes_{\q Q^m}\q)<\infty, \ \ \ 
\supdim \Big((W/U) \otimes_{\q Q^m}\q\Big)<\infty.
\] 
Now from the exact sequence $0\arr U \arr W \arr W/U\arr 0$, 
we obtain the long exact sequence 
\[
\cdots \arr \tor_1^{\q Q^m}(W/U,\q) \arr  U 
\otimes_{\q Q^m}\q  \arr W \otimes_{\q Q^m}\q 
\arr (W/U) \otimes_{\q Q^m}\q \arr 0,
\] 
which implies that
\begin{equation}\label{2}
\dim_\q (U \otimes_{\q Q^m}\q) \leq \dim_\q 
\tor_1^{\q Q^m}(W/U,\q)+\dim_\q (W \otimes_{\q Q^m}\q).
\end{equation}
Since $\supdim \Big((W/U) \otimes_{\q Q^m}\q)<\infty$, by Theorem \ref{d-f2} we obtain 
\begin{equation}\label{3}
\supdim H_i(Q^m,W/U)<\infty.
\end{equation}
But $\tor_i^{\q Q^m}(W/U,\q)=H_i(Q^m,W/U)$,
thus by (\ref{2}) and (\ref{3}) we have
\[
\supdim (U \otimes_{\q Q^m}\q)<\infty.
\]
\end{proof}

The next theorem is the main result of this paper.

\begin{thm}\label{d-f4}
Let $N \tail G \two Q$ be an exact sequence of groups, where $G$ is finitely generated, 
$N$ is nilpotent of class $c$ and $Q$ is abelian. If $N/N'$ is $2(c(n-1)+1)$-tame, 
then for any $ 0 \leq j \leq n$, $\vb_j(G)$ is finite.
\end{thm}
\begin{proof}
Let $G_1$ be a subgroup of finite index in $G$. Let $Q_1$ be the image of $G_1$ in $Q$ and
$N_1:=N\cap G_1$. Then clearly $[Q:Q_1]< \infty$, and $[N:N_1]< \infty$. From the associated 
Lyndon-Hochschild-Serre spectral sequence
\[
E^2_{p, q}=H_p(Q_1,H_q(N_1,\q))\Rightarrow H_{p+q}(G_1,\q)
\]
of the extension $ N_1 \tail G_1 \two Q_1$, we obtain
\[
\dim_{\q} H_j(G_1,\q)=\sum_{p=0}^j \dim_{\q}E_{p,j-p}^{\infty}\leq
\sum_{p=0}^j \dim_{\q}E_{p,j-p}^2.
\]
Since  $[N:N_1]< \infty$, by Corollary \ref{nil-coristriction}, for any $k \geq 0$, we have
\[
H_k(N_1,\q)\simeq H_k(N,\q).
\]
Thus $E^2_{p,q}\simeq H_p(Q_1,H_q(N,\q))$. On the other hand, since $[Q:Q_1]< \infty$, there
exists $m \in \N$ such that $(Q/Q_1)^m=1$. Hence $Q^m \subseteq Q_1$. Since $Q_1 / Q^m$ 
is finite, we have
$$
H_p(Q_1,H_{j-p}(N,\q)) \simeq H_p(Q^m,H_{j-p}(N,\q))_{Q_1/ Q^m}
$$
and this implies that
\[
\dim_\q H_p(Q_1,H_{j-p}(N,\q))\leq \dim_\q H_p(Q^m,H_{j-p}(N,\q)).
\]
So to prove the theorem it is sufficient to prove that
\[
\supdim H_p(Q^m, H_{j-p}(N, \q))< \infty.
\]
By Corollary \ref{nil-filtration-2}, $H_{j-p}(N,\q)$ has a natural 
filtration of $\q Q$-submodules
\[
0=E_0 \se E_1 \se \cdots \se E_{l-1} \se E_l= H_{j-p}(N,\q),
\]
such that for any $0\leq k \leq l,$ $E_k/E_{k-1}$ is a natural
subquotient of a $\q Q$-module from the set 
$\{\bigotimes^s_\q V\}_{0 \leq s \leq c(j-p-1)+1}$, where  $V:=(N/N')\otimes_\Z \q$ and  
$\bigotimes^s_\q V$ is considered as a $\q Q$-module via the diagonal action of $Q$. By 
Theorem \ref{b-g}, $\bigotimes^s_\q V$ is a finitely generated $\q Q$-module for 
$0 \leq s \leq 2c(j-p-1)+2$. Thus by Theorem \ref{b-k2013}, 
\[
\begin{array}{c}
\sup_{m \geq 1}\dim_\q \Big((\bigotimes^s_\q V)\otimes_{\q Q^m}\q\Big) <
\infty\ {\rm  for} \ 0 \leq s \leq c(j-p-1)+1.
\end{array}
\]
Next Lemma \ref{subquotient} implies that
\[
\supdim \Big((E_i/E_{i-1}) \otimes_{\q Q^m}\q\Big) < \infty,
\]
and by induction on $i$, one can show that, for any $1 \leq i \leq j-p$
\[
\supdim (E_i \otimes_{\q Q^m}\q) < \infty.
\]
Therefore
\[
\supdim \Big(H_{j-p}(N,\q) \otimes_{\q Q^m}\q\Big)=\supdim (E_l \otimes_{\q Q^m}\q) < \infty.
\]
Now by Theorem \ref{d-f2}, for any $0 \leq p\leq j$,
\[
\supdim H_p(Q^m, H_{j-p}(N, \q))< \infty.
\]
 This completes the proof of the theorem.
\end{proof}

\begin{lem}\label{5}
Let $G$ be a group and $H$ a subgroup of finite index in $G$. Then $\vb_n(G)$ is finite 
if and only if $\vb_n(H)$ is finite. In fact, for any $n\geq 0$, $\vb_n(G)=\vb_n(H)$.
\end{lem}
\begin{proof}
If $H_0$ is a subgroup of finite index in $H$, then $[G:H_0]=[G:H][H:H_0]<\infty$. So
$\dim_{\q}H_n(H_0,\q)\leq \vb_n(G)$ and hence  
\[
\vb_n(H)\leq \vb_n(G).
\]
If $G_0$ is a subgroup of finite index in $G$, then $[G_0:G_0\cap H]\leq[G:H]$. 
So there is a normal subgroup $N$ of $G_0$ such that $N\se G_0\cap H$ and 
$[G_0:N]<\infty$. Since $H_n(G_0,\q)\simeq H_n(N,\q)_{G_0/N}$,
$\dim_\q H_n(G_0,\q)\leq \dim_\q H_n(N,\q)$. Now from $[H:N]<\infty$, it follows 
that $\dim_\q H_n(G_0,\q) \leq \dim_\q H_n(N,\q)\leq \vb_n(H)$. Therefore 
\[
\vb_n(G)\leq \vb_n(H).
\]
\end{proof}

\begin{cor}\label{d-f5}
Let $G$ be a nilpotent-by-abelian-by-finite group, i.e. we have a chain of subgroups 
$N \unlhd H \unlhd G$, where $N$ is nilpotent, $H/N$ is abelian and $[G:H]<\infty$. If $N$
is of class $c$ and $H/N'$ is of type $\FP_{2c(n-1)+2}$, then $\vb_j(G)$ is finite for 
any $0\leq j\leq n$.
\end{cor}
\begin{proof}
Since $H/N'$ is metabelian of type $\FP_{2c(j-p-1)+2}$, by Theorem~\ref{b-g2} the $Q$-module
$(N/N')\otimes_\Z\q$ is $2(c(j-p-1)+1)$-tame. Now the claim follows from Lemma~\ref{5} and 
Theorem~\ref{d-f4}.
\end{proof}

\begin{rem}
Theorem \ref{d-f4} and Corollary \ref{d-f5} generalize
\cite[Theorem 5.3 and Corollary 5.4]{bridson-kochloukova2013}
to higher homology groups.
\end{rem}

For the first virtual rational Betti number we can improve the above result a bit.

\begin{prp}\label{betti1}
Let $N \tail G \two Q$ be an exact sequence of groups, where 
$N$ is nilpotent and  $Q$ is polycyclic. Let $G/N'$ be of type $\FP_3$ and let
$\bigotimes_\Z^2 N/N'$ be finitely generated 
as $\Z Q$-module via the diagonal action. Then $\vb_1(G)$ is finite. 
\end{prp}
\begin{proof}
Let $G_1$ be a normal subgroup of finite index in $G$. Let $Q_1$ be the image of the $G_1$ in
$Q$ and $N_1=N \cap G_1$. The associated Lyndon-Hochschild-Serre spectral sequence of
$N_1\tail G_1 \two Q_1$, i.e.
\[
E^2_{p,q}=H_p(Q_1,H_q(N_1,\q))\Rightarrow H_{p+q}(G_1,\q),
\]
implies that 
\begin{align*}
\dim_\q H_1(G_1,\q) &\leq \dim_\q E^2_{0,1}+\dim_\q E^2_{1,0}\\
&= \dim_\q H_0(Q_1,H_1(N_1,\q)) +\dim_\q H_1(Q_1,\q).
\end{align*}
Since any subgroup of a polycyclic group is polycyclic, by 
\cite[Lemma~3.2]{kochloukova-mokari2014} we have $\dim_\q H_1(Q_1,\q)\leq h(Q)$, 
where $h(Q)$ is the Hirsch length of $Q$.
Since $[N:N_1]<\infty$, by Corollary \ref{nil-transfer} we have
$H_1(N_1,\q)\simeq H_1(N,\q)$.
So to prove the claim it is sufficient to prove that 
\[
\sup_{[Q:Q_1]<\infty}\dim_\q (N/N'\otimes_{Q_1} \q) < \infty.
\]
Let $A=N/N'$ and $H=G/N'$ and consider  the exact sequence $A \tail H \two Q$.
If we put $A_0=[A,H]$ and $Q_0=H/A_0$ and if we follow the proof of Theorem~A in 
\cite{kochloukova-mokari2014},
we obtain
\[
\sup_{[Q_0:Q_2]<\infty}\dim_\q (A_0\otimes_{Q_2} \q) <\infty.
\]
From the exact sequence $ A_0 \tail A \two A/A_0$, we obtain the exact
sequence
\[
A_0\otimes_{Q_2}\q \arr A \otimes_{Q_2}\q \arr (A/A_0)\otimes_{Q_2}\q \arr 0,
\]
which implies that
\[
\dim_\q  (A \otimes_{Q_2} \q) \leq \dim_\q (A_0\otimes_{Q_2}\q)+ 
\dim_\q \Big((A/A_0)\otimes_{Q_2}\q\Big).
\]
Now consider the exact sequence $A/A_0 \tail Q_0 \overset{\beta}{\two} Q$ and 
let $Q_1=\beta(Q_2)$. Since the action of $A/A_0$ over $A$ is trivial, we have 
$A\otimes_{Q_1} \q \simeq A \otimes_{Q_2}\q$. Since $A/A_0$ is a finitely generated 
abelian group, 
\[
\sup_{[Q_0:Q_2]<\infty}\dim_\q \Big((A/A_0)\otimes_{Q_2} \q\Big) <\infty.
\]
Therefore from the above relations we have 
\[
\sup_{[Q:Q_1]<\infty}\dim_\q (A\otimes_{Q_1} \q) <\infty.
\]
This completes the proof of the theorem.
\end{proof}

\begin{cor}
Let $N \tail G \two Q$ be an exact sequence of groups, where $N$ is nilpotent and 
$Q$ is nilpotent of class $c \leq 2$. If $G/N'$ is of type $\FP_3$,
then $\vb_1(G)$ is finite.
\end{cor}
\begin{proof}
By Lemma 3.5 in the proof of Corollary~B in \cite{kochloukova-mokari2014}, 
$\bigotimes_\q^2 (A_0\otimes_\Z \q)$ is finitely generated as $\q Q$-module via the diagonal 
action, where $A_0$ is as in the proof of Theorem \ref{betti1}. Now we can proceed as
in the proof of Theorem \ref{betti1}.
\end{proof}

\section{Some examples}

\subsection{S-arithmetic groups}\label{koch}
Unfortunately there is no classification of the nilpotent-by-abelian groups of type 
$\FP_n$ even in the case of $n =2$, though the metabelian case was solved in \cite{bieri-strebel1980}.  
In this case type $\FP_2$ turns out to be equivalent to finite presentability. Still in the case of 
soluble $S$-arithmetic groups there is a complete classification of finite 
presentability \cite[Theorem~7.5.2, Remark~4, Chap. VII]{Abels}. They are finitely presented if 
and only if are of type $\FP_2$. Note that soluble $S$-arithmetic groups are 
nilpotent-by-abelian-by-finite.  

By a theorem of Borel-Serre \cite[Theorem~0.4.4]{Abels}, any $S$-arithmetic subgroup 
of a reductive group is of type $\FP_\infty$ and thus for such soluble subgroups the result of 
Corollary \ref{d-f5} is true for any $j\geq 0$. But such a result can be proved for other
type of $S$-arithmetic groups. 

The following example was considered in \cite{Abels-Brown}: Let $p$ be a prime and
\[
\Gamma_n \leq \GL_{n+1}(\mathbb{Z}[1/p]),
\]
where $\Gamma_n$ is the group of upper triangular matrices $A$ with 
$A_{1,1} = 1 = A_{n+1, n+1}$.

\begin{thm} 
The group $\Gamma_n$ is of type $\FP_{n-1}$, but not of type $\FP_n$.
\end{thm}
\begin{proof}
See Theorem A in \cite{Abels-Brown}.
\end{proof}

Let $N_n$ be the subgroup of $\Gamma_n$ containing all elements of $\Gamma_n$, where 
the main diagonal contains only entries $1$. Then $N_n$ is nilpotent and 
\[
Q_n = \Gamma_n / N_n \simeq \mathbb{Z}^{n-1}.
\]
In this case the abelianization 
$V_n = N_n / [N_n, N_n]$  is isomorphic to $\mathbb{Z}[{1/p }]^{n}$, so 
$V_n \otimes_{\mathbb{Z}} \mathbb{Q} \simeq \mathbb{Q}^{n}$ is finite dimensional 
over $\q$. Hence all tensor and exterior powers of $V_n$ are finitely generated over 
$\q Q_n$. Thus Theorem \ref{b-g} implies that $V_n\otimes_\Z\q$ is $m$-tame for any 
$m\geq 2$. Now by Theorem~\ref{d-f4} we obtain the following result.

\begin{prp} 
For any $j \geq 0$, $\vb_j(\Gamma_n)$ is finite.
\end{prp}

\subsection{Groups of finite torsion-free rank}
It is a well-known theorem of Mal'cev that polycyclic groups are 
nilpotent-by-abelian-by-finite \cite[3.1.14]{lennox-robinson2004}. 
On the other hand, for a polycyclic group $G$, the group ring $\mathbb{Z} G$
is (right) Noetherian \cite[4.2.3]{lennox-robinson2004} and thus $G$ is 
of type $\FP_\infty$. Now by Corollary \ref{d-f5}, all virtual rational Betti numbers
of $G$ are finite. A direct and much easier proof of this fact is given
in \cite[Lemma~3.2]{kochloukova-mokari2014}

A polycyclic group is a special case of constructible groups. 
A soluble group is called constructible if and only if it can be built from the trivial group 
in finitely many steps by taking descending HNN-extensions and finite extensions.
It is well-known that the class of constructible soluble groups is closed with respect 
to taking homomorphic images and subgroups of finite index 
\cite[Proposition~2, Theorem~4]{baumslag-bieri1976}. Moreover, they have finite Pr\"ufer rank
\cite[3.3, Remark~2]{baumslag-bieri1976} and thus are nilpotent-by-abelian-by-finite. The last 
part follows from the proof of \cite[Theorem~10.38]{robinson1972}.
Furthermore, constructible soluble groups are finitely presented and are of type $\FP_\infty$
\cite[Proposition~1]{baumslag-bieri1976}. Thus by Corollary~\ref{d-f5} all virtual rational Betti
numbers of these groups are finite.

Kochloukova and the second author gave a good bound for virtual rational Betti numbers of
a polycyclic group
\cite[Lemma~3.2]{kochloukova-mokari2014}. Their proof work even for the larger class of groups 
of finite torsion-free rank. Polycyclic and constructible groups are of finite Pr\"ufer rank 
and thus they are of finite torsion-free rank.

A group G, not necessarily soluble, is said to be of finite torsion-free rank if
it has a series of subgroups
\[
1=G_0 \lhd G_1 \lhd \dots \lhd G_n=G,
\]
such that each non-torsion factor $G_i/G_{i-1}$ is infinite cyclic. 
One can show that the number of infinite cyclic factors is independent 
of the chosen series  (see the proof of \cite[1.3.3]{lennox-robinson2004})
which it is called either the torsion-free
rank or the Hirsch number of $G$ and we denote it by $h(G)$. 

\begin{prp}
Let $G$ be a group of finite torsion-free rank. Then for any integer $j\geq 0$,
$\dim_\q H_j(G, \q)\leq \binom{h(G)}{j}$. In particular, 
\[
\vb_j(G)\leq \binom{h(G)}{j}.
\]
\end{prp}
\begin{proof}
The proof is similar to the proof of the case of polycyclic groups given
in \cite[Lemma~3.2]{kochloukova-mokari2014}.
\end{proof}


\bigskip
\address{{\footnotesize

Behrooz Mirzaii,

Institute of Mathematics and Computer Sciences (ICMC),

University of Sao Paulo (USP), Sao Carlos, Brazil.

e-mail:\ bmirzaii@icmc.usp.br,

}}

\bigskip
\address{{\footnotesize

Fatemeh Yeganeh Mokari,

Institute of Mathematics, Statistics and Scientific Computing (IMECC),

State University of Campinas (Unicamp), Campinas, Brazil.

email:\ f.mokari61@gmail.com

}}
\end{document}